\documentclass[10pt]{article}
\usepackage[utf8]{inputenc}
\usepackage{authblk}
\usepackage{amssymb}
\usepackage{amsmath}
\usepackage{amsthm}
\usepackage{mathrsfs}
\usepackage[margin=1in]{geometry}

\usepackage[nottoc]{tocbibind}
\usepackage{mathtools}

\newtheorem{thm}{Theorem}[section]
\newtheorem{prop}[thm]{Proposition}
\newtheorem{cor}[thm]{Corollary}
\newtheorem{lemma}[thm]{Lemma}

\theoremstyle{definition}
\newtheorem{defn}[thm]{Definition}

\newcommand\restr[2]{{
  \left.\kern-\nulldelimiterspace #1
  \vphantom{\big|}
  \right|_{#2}
  }}

\DeclarePairedDelimiter\floor{\lfloor}{\rfloor}

\newcommand{\fr}{\mathfrak}
\newcommand{\bb}{\mathbb}  
\newcommand{\Z}{\bb Z}
\newcommand{\Q}{\bb Q}

\newcommand{\C}{\bb C}
\newcommand{\N}{\bb N}
\newcommand{\Fp}{\bb F_p}

\newcommand{\e}{{\mathbf{e}}}

\newcommand{\isom}{\xrightarrow{\sim}}
\renewcommand{\rm}{\textrm}

\newcommand{\Span}{\normalfont\rm{Span}}
\newcommand{\sign}{\rm{sgn}}

\newcommand{\Sym}{\normalfont\rm{Sym}}
\renewcommand{\mod}{\normalfont\rm{ mod }}

\newcommand{\DP}{\mathcal{DP}}
\newcommand{\reg}{\rm{reg}}
    \newcommand{\End}{\normalfont\rm{End}}

\newcommand{\Hom}{\normalfont\rm{Hom}}
\newcommand{\Fun}{\normalfont\rm{Fun}}

\newcommand{\Ima}{\normalfont\rm{Im }}
\newcommand{\Int}{\normalfont\rm{Int}}
\newcommand{\Diff}{\normalfont\rm{Diff}}

\renewcommand{\geq}{\geqslant}
\renewcommand{\leq}{\leqslant}
\usepackage{titling}
\usepackage{lipsum}

\author{Daniil Kalinov, Lev Kruglyak}
\title{The Rational Cherednik Algebra of Type $A_1$ with Divided Powers}
\date{February 2020}

\usepackage{hyperref}
\usepackage{cleveref}
\usepackage{lipsum}
\usepackage{times}
\usepackage{etoc}
\usepackage{cite}
\begin{document}
\maketitle
\abstract{
Motivated by the recent developments of the theory of Cherednik algebras in positive characteristic, we study rational Cherednik algebras with divided powers. In our research we have started with the simplest case, the rational Cherednik algebra of type $A_1$. We investigate its maximal divided power extensions over $R[c]$ and $R$ for arbitrary principal ideal domains $R$ of characteristic zero. In these cases, we prove that the maximal divided power extensions are free modules over the base rings, and construct an explicit basis in the case of $R[c]$. In addition, we provide an abstract construction of the rational Cherednik algebra of type $A_1$ over an arbitrary ring, and prove that this generalization expands the rational Cherednik algebra to include all of the divided powers.}

\tableofcontents

\section{Introduction}\label{introduction}
In this paper we study the rational Cherednik algebra of type $A_{n-1}$, which we denote by $\mathcal H_{t,c}(\fr S_n, \frak h)$. Cherednik algebras, also known as double affine Hecke algebras (DAHA), are a large family of algebras introduced by Cherednik in \cite{  cherednik1993macdonald} to prove Macdonald's conjectures concerning orthogonal polynomials for root systems. Since then Cherednik algebras have been discovered to be useful in many different contexts, most notably in the study of quantum Calogero-Moser systems (see \cite{etingof2007calogero}). Cherednik algebras have also been applied to topology, harmonic analysis, Verlinde algebras, Kac-Moody algebras and more. For a thorough exposition of theory of DAHA in general, see \cite{cherednik2005double}. Another good overview of the theory of rational Cherednik algebras is given in \cite{etingof2010lecture}.

The representation theory of Cherednik algebras over fields of characteristic zero has been well studied (see \cite{gordon2003baby}, \cite{etingof2010lecture}), but more recently a theory of Cherednik algebras in positive characteristic started to develop. Cherednik algebras in positive characteristic were investigated in  \cite{balagovic2013representations} and \cite{bezrukavnikov2006cherednik}. In \cite{latour2005representations}, the case of rank one algebras was discussed. Later in \cite{devadas2014representations}, \cite{devadas2016polynomial}, and \cite{cai2018cherednik} the Hilbert polynomials of some irreducible finite dimensional representations were calculated.

The current paper is a continuation of this research. Our main goal was to develop a theory of Cherednik algebras with divided powers in positive characteristic, so we have started with the simplest example, the rational Cherednik algebra of type $A_1$. To define the maximal divided power extension even in this case turned out to be an interesting problem. For more information on algebras with divided powers see \cite{jantzen2007representations} and \cite{lonergan2016strong}. The main reason for the study of this construction is the fact that naive reduction of the Cherednik algebra to positive characteristic makes the algebra ``too small", because a lot of operators become central and act by zero on important representations. To make representation theory richer one can work with the algebra extended by divided powers.

\begin{subsection}{Main Results}
In Section~\ref{introduction}, we define the rational Cherednik algebra of type $A$, introduce our notion of divided power extensions, and show an example of this notion applied to an algebra of differential operators. In Section~\ref{maximaldivpowexts}, we prove Theorem~\ref{freenessofdividedpowerextensiontheorem} and Theorem~\ref{basistheorem} which show the freeness of the maximal divided power extension of the rational Cherednik algebra of type $A_1$ over $R$ and $R[c]$, constructing a basis in the latter case. In Section~\ref{abstractconstructionsection}, we construct the maximal divided power extension in an abstract way over an arbitrary ring, and prove equivalence in most cases.
\end{subsection}

\subsection{The Rational Cherednik Algebra of Type $A$}\label{therationalcherednikalgebraoftypeA}
In this section we will define the rational Cherednik algebra of type $A_{n-1}$, which we denote $\mathcal H_{t,c}(\fr S_n, \fr h)$. In general we will work with the rational Cherednik algebra over an arbitrary ring, but here we introduce the standard notion over the field of complex numbers. Let $\fr S_n$ be the symmetric group on $n$ elements and consider its permutation representation on $\fr h = \C^n$ and its dual $\fr h^\ast$. For any $1\leq i \neq j \leq n$, let $s_{ij}\in \fr S_n$ denote the reflection switching $i$ and $j$. For each reflection $s_{ij}$, let $P_{ij}\subset \fr h$ be the hyperplane of fixed points of $s_{ij}$, i.e. $P_{ij} = \{ (\alpha_1,\ldots, \alpha_n) : \alpha_i=\alpha_j \}$. Let $\fr h_\reg = \fr h \setminus \bigcup_{i<j}P_{ij}$ be the set of regular points of $\fr h$, i.e. the set of points which are not fixed by any reflection. Let $\mathcal D(\fr h_\reg)$ be the algebra of differential operators on the set $\fr h_\reg$. We have a natural action of $\fr S_n$ on $\fr h_\reg$ and hence on $\mathcal D(\fr h_\reg)$. Note that $\mathcal D(\fr h_\reg)$ is isomorphic to the localization $\{x_i-x_j\}_{i\neq j}^{-1}\Diff(\C[\fr h])$ where $x_1,\ldots,x_n$ are the standard generators of $\C[\fr h]$. The following results and definitions are taken from \cite{etingof2010lecture}.

\begin{defn}
For any $1\leq i \leq n$ and $t,c\in \C$, the Dunkl operator is defined as
\[D_i = t\dfrac{\partial}{\partial x_i} - c\sum_{j\neq i}\dfrac{1}{x_i-x_j}(1-s_{ij}) \in \mathcal D(\fr h_\reg)\rtimes \C[\fr S_n].\]
\end{defn}

\begin{prop} We have the following properties for Dunkl operators:
\begin{itemize}
    \item For $\sigma \in \fr S_n$, we have $\sigma D_i \sigma^{-1} = D_{\sigma(i)}$
    \item $[D_i,D_j]=0$
    \item $[D_i, x_j] = cs_{ij}$
    \item $[D_i, x_i] = t-c\sum_{j\neq i}s_{ij}$
\end{itemize}
\end{prop}

We can now define the rational Cherednik algebra of type $A$.

\begin{defn}
For any $t,c\in \C$ with $t\neq 0$, let $\mathcal H_{t,c}(\fr S_n, \fr h)$ be the $\C$-subalgebra of $\mathcal D(\fr h_\reg)\rtimes \C[\fr S_n]$ generated by $\fr h^\ast$, $\fr S_n$ and $D_i$ for $i=1,\ldots,n$. This is the rational Cherednik algebra of type $A_{n-1}$ associated to $t,c$.
\end{defn}

\begin{prop}
For any $t,c\in \C$ with $t\neq 0$, the algebra $\mathcal H_{t,c}(\fr S_n,\fr h)$ is isomorphic to the quotient of the algebra $\C\langle x_1,\ldots,x_n,y_1,\ldots, y_n\rangle\rtimes \C [\fr S_n]$ by the relations
\[[x_i,x_j] = 0,\quad [y_i,y_j]=0, \quad [y_i,x_j]=cs_{ij},\quad [y_i,x_i] = t-\sum_{j\neq i}cs_{ij}.\]
\end{prop}

\begin{thm}[PBW Theorem] 
Let $\Sym(V)$ be the symmetric algebra of $V$. Let $x_1,\ldots, x_n$ be the standard basis for $\fr{h}^\ast$ and let $y_1,\ldots, y_n$ be the corresponding basis of $\fr{h}$. Then the map
\[\Sym(\fr h)\otimes_\C \C[\fr S_n] \otimes_\C \Sym(\fr h^\ast) \to \mathcal H_{t,c}(\fr S_n, \fr h), \]
which sends $y_i\otimes g \otimes x_i \mapsto D_igx_i$, is an isomorphism of $\C$-vector spaces.
\end{thm}

There is another useful algebra to consider when studying divided power extensions of $\mathcal H_{t,c}(\fr S_n, \fr h)$. Consider the permutation representation of $\fr S_n$ on $\fr{h}$ and its dual $\fr{h}^\ast$, with bases $y_1,\ldots, y_n$ and $x_1,\ldots, x_n$ respectively. Consider the subrepresentation $\fr l = \Span_{\C} \{ \hat{y}_i = y_i-y_1 : 1<i\leq n\}$ and its dual $\fr l^\ast = {\fr h}^\ast / \langle x_1+x_2+\cdots+x_n\rangle$. Let $\mathcal T(\fr l\oplus \fr l^\ast)$ be the tensor algebra of $\fr l\oplus \fr l^\ast$.

\begin{defn}
 $\mathcal H_{t,c}(\fr S_n, \fr l)$ is the $\C$-subalgebra of $\End(\Sym(\fr l^\ast))$ generated by $\fr l^\ast, \fr S_n$ and $D_{i}-D_{1}$.
\end{defn}

\begin{prop}
The algebra $\mathcal H_{t,c}(\fr S_n, \fr l)$ is the quotient of $\mathcal T(\fr l\oplus \fr l^\ast)\rtimes \C[\fr S_n]$ by the relations:
\begin{itemize}
    \item $[x_i,x_j]=0$
    \item $[\hat{y}_i, \hat{y}_j] = 0$
    \item $[\hat{y}_i,x_i] = t-cs_{1i} - c\sum_{k\neq i}s_{ik}$
    \item $[\hat{y}_i,x_k] = cs_{ik} - cs_{1k}$ for $k\neq i,1$
\end{itemize}
\end{prop}

The algebras $\mathcal H_{t,c}(\fr S_n, \fr h)$ and $\mathcal H_{t,c}(\fr S_n, \fr l)$ are related in the following way. Let  $z_1=y_1-y_2$,\linebreak $z_2 = y_2-y_3, \ldots z_{n-1}=y_1-y_n$ and $Z= y_1+\cdots + y_n$. Let $w_1=x_1-x_2$, $w_2 = x_2-x_3, \ldots, w_{n-1}=x_1-x_n$ and $W= x_1+\cdots + x_n$. Note that $[Z,x_i] = t$ and $[W,y_i] = -t$, it follows that $[Z,w_i]=[W,z_i]=0$. Also $[Z,W] = n$. Furthermore, $[\sigma,Z] = [\sigma,W]=0$ for all $\sigma\in \fr S_n$. So we have two subalgebras, one generated by $z_1,\ldots, z_{n-1}$, $w_1,\ldots, w_{n-1}$ and $\fr S_n$ and the other generated by $Z$ and $W$. The first algebra is isomorphic to $\mathcal H_{t,c}(\fr S_n, \fr l)$, and the second algebra is isomorphic to $\C[q,{\partial_q}]$, the subalgebra of $\End(\C[q])$ generated by $q$ and $\frac{\partial}{\partial q}$ for some formal variable $q$. By the PBW theorem, it follows that
\[\mathcal H_{t,c}(\fr S_n, \fr h) \cong \mathcal H_{t,c}(\fr S_n, \fr l)\otimes_\C \C[q,\partial_q].\]

Another useful algebra to consider is the spherical subalgebra of $\mathcal H_{t,c}(\fr S_n, \fr h)$, denoted by $\mathcal B_{t,c}(\fr S_n, \fr h)$.

\begin{defn}
Let $\e_+\in \C[\fr S_n]$ be the symmetrizer, $\e_+ = \frac{1}{n!}\sum_{\sigma\in \fr S_n}\sigma$. Let $\e_-$ be the antisymmetrizer, $\e_- = \frac{1}{n!}\sum_{\sigma\in \fr S_n}\sign(\sigma) \sigma$ where $\sign(\sigma)$ is the sign of a permutation.
\end{defn}

\textit{Note: } $\e_+^2 = \e_+$ and $\e_-^2=\e_-$.

\begin{defn}
The spherical subalgebra of $\mathcal H_{t,c}(\fr S_n, \fr h)$ is $\mathcal B_{t,c}(\fr S_n, \fr h) = \e_+ \mathcal H_{t,c}(\fr S_n, \fr h) \e_+$. Let $\mathcal B_{t,c}(\fr S_n, \fr l) = \e_+ \mathcal H_{t,c}(\fr S_n, \fr l) \e_+$.
\end{defn}
Note that $\e_+ (\mathcal D(\fr h_\reg)\rtimes \C[\fr S_n] )\e_+ = \mathcal D(\fr h_\reg)^{\fr S_n}$, i.e. the $\fr S_n$-invariant subspace of $\mathcal D(\fr h_\reg)$. This means that $\mathcal B_{t,c}(\fr S_n, \fr h) \subset \mathcal D(\fr h_\reg)^{\fr S_n}$. Since $\fr S_n$ acts trivially on $\C[q,\partial_q]$, we have the decomposition
\[\mathcal B_{t,c}(\fr S_n, \fr h) \cong \mathcal B_{t,c}(\fr S_n, \fr l)\otimes_\C \C[q,\partial_q].\]
\subsection{Divided Power Extensions}\label{dividedpowerextensions}
We could not find a definition of divided powers in the existing literature which worked for our purposes, so we have developed our own framework.

Let $R$ be an integral domain of characteristic zero and let $V$ be a free $R$-module. Note that we have a canonical embedding $\End_R(V) \hookrightarrow \End_{R\otimes \Q}(V\otimes \Q)$\footnote{Unless stated otherwise, all tensor are assumed to be taken over $\Z$.}. 

\begin{defn}\label{defdividedpows}
For any submodule $A\subset \End_R(V)$, the maximal divided power extension of $A$, denoted $A^\DP$, is the submodule of $\End_{R\otimes\Q}(V\otimes \Q)$ given by:
\[A^\DP = (A\otimes \Q)\cap \End_R(V)\subset \End_{R\otimes \Q}(V\otimes \Q)\]
\end{defn}

Note that $A^\DP$ is an $R$-module, and if $A$ is an $R$-algebra, then $A^\DP$ is an $R$-algebra as well. Another insightful definition of $A^\DP$ arises through the notion of divisibility of an operator.

\begin{defn}\label{dividedpowerdefinition2}
For some operator $f\in \End_R(V)$, and integer $n\in \Z_{\ge 1}$, we say that $n$ divides $f$ if $f\otimes (1/n) \in \End_R(V)$. We write $n|f$. 
\end{defn}
The following definition is often easier to use in practice than Definition~\ref{defdividedpows}.
\begin{prop}
$A^\DP = \{ f\otimes {(1/n)} : f\in A, n\in \Z_{\geq 1}, n|f\}$.
\end{prop}

\begin{proof}
This follows from the fact that every $a\in A\otimes \Q$ can be uniquely expressed as $f\otimes (1/n)$, for some $f\in A$ and $n\in \Z$.
\end{proof}

To show how this notion of divided power extensions applies to representation theory in characteristic $p$, suppose we had some faithful representation $\psi : A \to \End_R(V)$. The naive reduction modulo $p$ gives a representation $A\otimes \Fp \to \End_{R\otimes \Fp}(V\otimes \Fp)$. The center of $A\otimes \Fp$ can become large in characteristic $p$. Since central operators may act trivially on $V\otimes \Fp$, this can become problematic. If we instead take the divided power extension, we have a representation $A^\DP\otimes \Fp \to \End_{R\otimes \Fp}(V\otimes \Fp)$. This representation is faithful, since if the image of $Q\otimes 1$ was zero, then $Q=p^nL$ for some $n\geq 1$ and $L\in A^\DP$ such that $L\otimes 1\neq 0$ in $A^\DP\otimes\Fp$. This means that $Q\otimes 1=0$ in $A^\DP\otimes \Fp$, so the map is injective. In the cases when $R^\times \cap \Z = \{\pm 1\}$, $A^\DP\otimes \Fp$ contains a nonzero scaled copy of each nonzero operator in $A$. This can make the representation theory of $A^\DP$ richer than that of $A$ in characteristic $p$.

When computing maximal divided power extensions of a ring, it often helps to decompose the ring into smaller pieces for which the maximal divided power extensions are already known.

\begin{prop}\label{directsumdividedpowers}
Suppose $\{A_i\}_{i\in I}$ is a family of $R$-submodules of $\End_R(V)$ and suppose that for any $a_i\in A_i$, $d|\sum_{i\in I}a_i$ in $\End_R\left(V\right)$ implies that $d|a_i$ for all $i\in I$. Then, in $\End_{R\otimes \Q}\left(V\otimes \Q\right)$, we have $\left(\bigoplus_{i\in I}A_i\right)^\DP = \bigoplus_{i\in I}A_i^\DP$.
\end{prop}

\textit{Note:} The above divisibility condition implies that the sum $\bigoplus_{i\in I}A_i$ is direct.

\begin{proof}
For any $Q=\sum_{i\in I}a_i$, if $d|Q$ then by assumption, $d|a_i$ for all $i\in I$ so $\frac{Q}{d} = \sum_{i\in I}\frac{a_i}{d}\in \bigoplus_{i\in I}A_i^\DP$. So $\left(\bigoplus_{i\in I}A_i\right)^\DP \subset \bigoplus_{i\in I}A_i^\DP$. Conversely, if $Q=\sum_{i\in I}\frac{a_i}{d_i}\in \bigoplus_{i\in I}A_i^\DP$ we have
\[Q = \frac{\sum_{i\in I}a_i\prod_{j\in I, j\neq i}d_i}{\prod_{i\in I}d_i}.\]
So $Q\in \left(\bigoplus_{i\in I}A_i\right)^\DP$ and $\left(\bigoplus_{i\in I}A_i\right)^\DP \supset \bigoplus_{i\in I}A_i^\DP$, so the result follows.
\end{proof}

\begin{prop}\label{tensorproductdividedpowers}
Let $V$ and $W$ be free $R$ modules. Suppose that $A=\bigoplus_{i\in I} A_i\subset \End_R(V)$ and $B=\bigoplus_{j\in J} B_j\subset \End_R(W)$ satisfy the divisibility condition of Proposition~\ref{directsumdividedpowers}. Furthermore, suppose that $A_i\cong B_j\cong A_i^\DP \cong B_j^\DP\cong R$. Finally, we make the additional requirement that $R^\times \cap \Z =\{\pm 1\}$. Then in $\End_{R\otimes \Q}(V\otimes W\otimes \Q)$,
\[(A\otimes_R B)^\DP = A^\DP \otimes_R B^\DP.\]
\end{prop}

\begin{proof}
First we claim that if $d|\sum_{(i,j)\in I\times J}a_i\otimes b_j$ then $d|a_i\otimes b_j$. Let $x_i$ be the basis element for $A_i^\DP$ and let $y_j$ be the basis element for $B_j^\DP$. Write $a_i\otimes b_j = k_{ij}x_i\otimes y_i$. If $d|\sum_{(i,j)\in I\times J}a_i\otimes b_j$, by definition there exists $q_{ij}$ such that $\sum_{(i,j)\in I\times J}k_{ij}x_i\otimes y_j= d\sum_{(i,j)\in I\times J}q_{ij}x_i\otimes y_j$. This implies $\sum_{(i,j)}(k_{ij}-dq_{ij})(x_i\otimes y_i)=0$. By linear independence, $k_{ij}=dq_{ij}$ and so $d|a_i\otimes b_j$. 

Next we claim that $(A_i\otimes_R B_j)^\DP =  A_i^\DP\otimes_R B_j^\DP$. To show $A_i^\DP\otimes_R B_j^\DP \subset (A_i\otimes_R B_j)^\DP$, let $\frac{a_i}{d_i}\otimes \frac{b_j}{k_j}\in (A_i\otimes_R B_j)^\DP$ for some $a_i\in A_i$, $b_j\in B_j$, and $k_j,d_i\in \Z_{>0}$. Then
\[\frac{a_i}{d_i}\otimes\frac{b_j}{k_j} = \frac{a_ik_j\otimes b_jd_i}{d_ik_j}\in (A_i\otimes_R B_j)^\DP.\]
Now to show that $(A_i \otimes_R B_j)^\DP\subset A_i^\DP\otimes_R B_j^\DP$, suppose $d|a_i\otimes b_j=k_{ij}x_i\otimes y_j$ for some $d\in \Z_{\geq 1}$. For an operator $f$ on some space $Z$, let $N_f = \{ n : n | f(z) \text{ for some }z\in Z\}$. Note that $N_{x_i}\cdot N_{y_j} \subset N_{x_i\otimes y_j}$. We claim that $\gcd(N_{x_i})=\gcd(N_{y_j})=1$. Indeed, if $d|N_{x_i}$ for some $d\in \Z_{>0}$ then $\frac{1}{d}x_i \in A_i^\DP$ and so $\frac{1}{d}\in R$, a contradiction unless $d=1$. The same argument shows that $\gcd(N_{y_j})=1$. We claim that $\gcd(N_{x_i\otimes y_j})=1$. Indeed if $d|N_{x_i\otimes y_j}$, then $d|N_{x_i}N_{y_j}$. Pick some $\ell \in N_{y_j}$. Then $d|\ell N_{x_i}$, but since $\gcd(\ell N_{x_i})=\ell$ it follows that $d|\ell$. Since $\ell$ was arbitrary, $d|N_{y_j}$, which implies that $d=1$. Now since $d|k_{ij}x_i\otimes y_j$, we have $d|k_{ij}N_{x_i\otimes y_j}$, so by the previous argument, $d|k_{ij}$. So
\[\frac{a_i\otimes b_j}{d} = \frac{k_{ij}x_i\otimes b_j}{d} = \frac{k_{ij}}{d}x_i\otimes b_j \in A_i^\DP\otimes B_j^\DP.\]
Now to combine the above claims, we have
\[(A\otimes_R B)^\DP = \left(\bigoplus_{(i,j)\in I\times J}A_i\otimes_R B_j\right)^\DP = \bigoplus_{(i,j)\in I\times J}(A_i\otimes_R B_j)^\DP = \bigoplus_{(i,j)\in I\times J}A_i^\DP\otimes_R B_j^\DP = A^\DP \otimes_R B^\DP,\]
where the second equality follows by the first claim and Proposition~\ref{directsumdividedpowers}. This completes the proof.
\end{proof}

\subsection{Polynomial Differential Operators}\label{polynomialdifferentialoperators}
To show a known example of divided power extensions, we consider the integral Weyl algebra \linebreak $W(\Z) = \Z\langle x,y\rangle/(yx-xy-1)$ and its faithful polynomial representation in $\End(\Z[x])$ given by $x\mapsto x\times$ (i.e. multiplication by $x$) and $y\mapsto \partial_x$ where $\partial_x=\frac{\partial}{\partial x}$. Let $\Z[x, \partial_x]\subset \End(\mathbb Z[x])$ be the image of this representation. We call this the ring of integral polynomial differential operators. Similarly define $\Q[x,\partial_x]$. The results of this section aren't original. Nonetheless we decided to include their proofs, adapted to fit within our framework of divided power extensions, because they illustrate a simple example of the methods we use in the case of Cherednik algebra in Section~\ref{freeness}.

\begin{defn}
Let ${t\choose k}\in \Q[t]$ be the polynomial ${t\choose k} = \frac{t(t-1)\cdots (t-k+1)}{k!}\in\Q[t]$, and $P_k(t) = k! {t\choose k} \in \Z[t]$. Let $ \mathcal D_x^k$ be the Hasse derivative, whose action is given on the basis by $ \mathcal D_x^k x^n = {n\choose k}x^{n-k} = \frac{\partial_x^k}{k!}x^{n-k}$ and extending linearly.
\end{defn}

\begin{prop}[Newton's Interpolation Formula]
Define the zeroth order forward difference operator as $\Delta^0 f(n)=f(n)$, and define the higher order operators as $\Delta^kf(n)=\Delta^{k-1}f(n+1)-\Delta^{k-1}f(n)$.
Let $f(t)$ be a polynomial. Then $f(t) = \sum_{k\geq 0}{t\choose k} \Delta^kf(0).$
\end{prop}

\begin{lemma}\label{intvaldivlemma}
Let $f$ be some integer-valued polynomial, and write $f(n)=\sum_{k\geq 0}\alpha_k {n\choose k} $ for some integer coefficients $\alpha_k$. If $d|f(n)$ for all $n\in\mathbb{Z}_{\geq 0}$, then $d|\alpha_k$ for all $k\geq 0$.
\end{lemma}

\begin{proof}
Suppose $f(n) \equiv 0\mod d$ for all $n$. Let $N=\deg f$, so $\alpha_n=0$ whenever $n>N$. By Newton's Interpolation formula, $LA=F\equiv 0\mod d$ where $(L)_{ij} = {i\choose j}$ is the $(N+1)\times (N+1)$ lower triangular Pascal matrix, $A = (\alpha_0, \ldots, \alpha_N)$, and $F=(f(0), \ldots, f(N))$. Note that $\det L=1$. Multiplying both sides by $L^{-1}$, we get that $\alpha_k\equiv 0\mod d$ for all $0\leq k \leq N$. It follows that $\alpha_k\equiv 0\mod d$ for all $k\geq 0$.
\end{proof}

The above lemma implies the following classical result.

\begin{prop}[Newton]\label{newtonfree}
Let $\Int(\Z[x]) = \{f\in \Q[x] : f(\Z)\subset \Z\}$. Then $\Int(\Z[x])$ is a free $\Z$-module generated by the polynomials ${t\choose k}$.
\end{prop}

\begin{prop}\label{isomgrad}
For any $n\geq 0$, let $D[n] = \Z[x]$ and for $n< 0$, let $D[n]= P_{-n}(x)\Z[x]$. Consider the map $\psi_n : D[n]\to \End(\Z[x])$ where $f(t)\in D[n]$ is sent to the operator which acts on $x^t$ by sending it to $f(t)x^{t+n}$. There is an isomorphism of $\Z$-modules, $\psi : \bigoplus_{n\in \Z}D[n] \to \Z[x, \partial_x]$, 
where $\psi|_{D[n]} = \psi_n$ for all $n\in \Z$.
\end{prop}
\begin{proof}
Consider the $\Z$-grading on $\Z[x, \partial_x]$ given by $\partial_x \mapsto -1$ and $x\mapsto 1$. Let $P[n]$ be the set of homogeneous elements of degree $n$. Since $\{x^l\partial_x^k\}_{l,k\geq 0}$ is a basis for $\Z[x, \partial_x]$ as a $\Z$-module, we have an isomorphism $\Z[x, \partial_x] \cong \bigoplus_{n\in\Z}P[n]$. We claim that $\psi_n : D[n] \to P[n]$ is an isomorphism. First, note that $\Ima(\psi_n)\subset P[n]$. This is clear if $n\geq 0$. Indeed, let $f(x)\in D[n]=\Z[x]$ be some polynomial, say $f(x)=\sum^{d}_{i=0}\alpha_i x^i$. Then
\[\psi_n(f(x)) = x^n\sum^d_{i=0}\alpha_i (x\partial_x)^i \in P[n].\]
Similarly, if $n<0$, let $P_{-n}(x)f(x) \in D[n] = P_{-n}(x)\Z[x]$ be arbitrary, with $f(x)=\sum^d_{i=0}\alpha_i x^i$. Then $\psi_n(P_{-n}(x)f(x)) = \partial_x^{-n}\psi_0(f(x))\in P[n]$.

To show surjectivity, we consider the cases $n\geq 0$ and $n<0$ separately. If $n\geq 0$, this map is surjective, since $\psi_n(P_l(t)) = x^{l+n}\partial^l_x$, and $x^{l+n}\partial^l_x$ generate $P[n]$. If $n<0$, by the grading, every $Q\in P[n]$ can be expressed as $L\partial_x^{-n}$ for some $L\in P[0]$. So $\psi_n(\ell (x+n)P_{-n}(x)) = Q$ where $\ell(x)$ is the polynomial representing the action of $L$. Since $L\in P[0]$ is arbitrary, it follows that $\Ima(\psi_n) = \psi_n(P_{-n}(x)\Z[x]) = P[n]$. So for any $n\in \Z$, the map $\psi_n : D[n] \to P[n]$ is a surjection, hence an isomorphism. We have the desired isomorphism $\psi$ by the definition of direct sum.
\end{proof}

\begin{defn}
Let $R$ be an integral domain of characteristic zero, and suppose $A$ is a submodule of $\Fun(\Z, R)$, the $\Z$-module of set-theoretic functions from $\Z$ to $R$. The ring of $R$-valued elements of $A$ is
\[\Int_R(A) = \{ f/d : f\in A, d|f, d\in \Z_{\ge 1}\}.\]
where $d|f$ if $f/d\in \Fun(\Z,R)\subset \Fun(\Z,R\otimes \Q)$. Note that this agrees with our earlier definition of $\Int(\Z[x])$. We write $\Int(A)$ if $R=\Z$.
\end{defn}

\begin{prop}\label{D_Zspan}
We have an isomorphism of $\Z$-modules, $\Z[x, \partial_x]^\DP \cong \bigoplus_{n\in \Z} \Int(D[n])$. In particular, this implies that as a $\Z$-module, $\Z[x, \partial_x]^\DP$ is spanned by $x^k\mathcal D_x^l$ for all $k,l\geq 0$. Furthermore these are $\Z$-linearly independent.
\end{prop}

\begin{proof}
To apply Proposition~\ref{directsumdividedpowers}, we must prove the divisibility condition. So suppose $d|\sum_{n\in \Z}Q_n$ where $\deg Q_n = n$. Then for all $t\geq 0$, $d|\left(\sum_{n\in\Z}Q_n\right)x^t = \sum_{n\in \Z}f_{Q_n}(t)x^{t+n}$. Therefore $d|f_{Q_n}(t)$ for all $t\geq 0$, and so $d|Q_n$. So by Proposition~\ref{directsumdividedpowers}, we have the equality $\bigoplus_{n\in \Z}P[n]^\DP = \Z[x, \partial_x]^\DP$. Note however that $d|Q\in P[n]$, if and only if $d|q(t) \in D[n]$ for all $t$, where $q(t)$ is the polynomial representing the action of $Q$ on $x^t$. So we have an isomorphism $\psi_n^\DP : \Int(D[n])\to P[n]^\DP$ for each $n$, defined similarly to $\psi_n$. Combining these, we get an isomorphism $\psi^\DP : \Z[x, \partial_x]^\DP \cong \bigoplus_{n\in\Z}\Int(D[n])$.

Next we claim that $\Z[x,\partial_x]^\DP$ is generated by $x^k\mathcal D_x^l$ for all $k,l\geq 0$ as a $\Z$-module. It suffices to consider $\Int(D[n])$, so first assume that $n\geq 0$. By Corollary~\ref{newtonfree}, $\Int(\Z[t])$ is generated by ${t\choose k} $. So the image of $\Int(D[n])$ in $\Z[x, \partial_x]^\DP$ is generated by $x^{n+k}\mathcal D_x^k$, since $x^{n+k}\mathcal D_x^kx^t = {t\choose k} x^{t+n}$. If $n<0$, note that by Lemma~\ref{intvaldivlemma}, $\Int(D[n]) = \Int(P_{-n}(t)\Z[t])$ is generated by ${t\choose k-n} $ for $k\geq 0$. This is because if $d| P_{-n}(t)\sum_{j\geq 0}\alpha_j P_j(t-n) = \sum_{j\geq 0}\alpha_j (j-n)! {t\choose j-n} $, then $d|\alpha_j(j-n)!$ for all $j$. This basis for $\Int(P_{-n}(t)\Z[t])$ corresponds to $x^{k}\mathcal D_x^{k-n}$, where $k\geq 0$. The $\Z$-linear independence follows from linear independence of ${t\choose k-n}$ in $\Z[t]$.
\end{proof}

\begin{cor}
$\Z[x, \partial_x]^\DP$ is a free $\Z[x]$-module, freely spanned by $\mathcal{D}^k_{x}$ for $k\geq 0$.
\end{cor}

\section{Maximal Divided Power Extensions of $\mathcal H_{t,c}(\fr S_2, \fr h)$}\label{maximaldivpowexts}
To apply the notion of divided powers to $\mathcal H_{t,c}(\fr S_2, \fr h)$, we must introduce an integral version of this algebra. Before we do this, we use the tensor decomposition given in Section~\ref{therationalcherednikalgebraoftypeA} to reduce the size of the algebra. Let $n=2$, and consider $\mathcal H_{t,c}(\fr S_2, \fr h)$.
This is a subalgebra of $\{x_1-x_2\}^{-1}\Diff(\C[x_1,x_2])\rtimes \C[\fr S_2]$ generated by $x_1,x_2, s_{12}$, 
\[D_1 = t\frac{\partial}{\partial x_1} - c\frac{1}{x_1-x_2}(1-s_{12})\quad\rm{ and }\quad D_2 = t\frac{\partial}{\partial x_2} + c\frac{1}{x_1-x_2}(1-s_{12}).\]
$\mathcal H_{t,c}(\fr S_2, \fr l)$ is the subalgebra of $\End(\C[x])$ generated by $x, s$ and $t\frac{\partial}{\partial x}-\frac{2c}{x}\frac{(1-s)}{2}$ where $sx = -xs, s^2=1,$ and $s\frac{\partial}{\partial x} = - \frac{\partial}{\partial x} s$. Here $x=x_1-x_2$ and $s=s_{12}$. Note that $\mathcal H_{t,c}(\fr S_2, \fr h) \cong \mathcal H_{t,c}(\fr S_2, \fr l) \otimes \C[q, \partial_q]$. By definition, $\mathcal H_{t,c}(\fr S_2, \fr h) \subset \End(\C[\fr h]) = \End(\C[\fr l])\otimes\End(\C[q])$, where $q$ is some formal variable. Since $\mathcal H_{t,c}(\fr S_2, \fr l)\subset \End(\C[\fr l])$ and $\C[q, \partial_q]\subset\End(\C[q])$, Proposition~\ref{tensorproductdividedpowers} implies that to study divided power extensions of $\mathcal H_{t,c}(\fr S_2, \fr h)$, it suffices to study divided power extensions of $\mathcal H_{t,c}(\fr S_2, \fr l)$ and $\C[q, \partial_q]$ separately. The conditions of Proposition\ref{tensorproductdividedpowers} are shown to be satisfied by the results of Section~\ref{polynomialdifferentialoperators} and Section~\ref{maximaldivpowexts}.  Since divided power extensions of $\C[q, \partial_q]$ are known (see Section~\ref{polynomialdifferentialoperators}), we only need to consider $\mathcal H_{t,c}(\fr S_2,\fr l)$.

Using the canonical isomorphism $\mathcal H_{t,c}(\fr S_2,\fr l) \to \mathcal H_{\lambda t,\lambda c}(\fr S_2, \fr l)$ for any $\lambda\in \C^\times$, we can normalize $t=0$ or $t=1$. In this paper, we only consider the case when $t=1$.
\begin{defn}
For any domain of characteristic zero $R$ and $c\in R$, let $H_{1,c}(R)$ be the subalgebra of $\End_R(R[x])$ generated by $\e_-, x$ and $D=\frac{\partial}{\partial x}-\frac{2c}{x}\e_-$. Note that $\e_+ = 1-\e_-$. In particular note that $H_{1,c}(\C) = \mathcal H_{1,c}(\fr S_2, \fr l)$.
\end{defn}

\subsection{Freeness of $H_{1,c}^\DP(R)$}\label{freeness}
In this section, we prove the following theorem:
\begin{thm}\label{freenessofdividedpowerextensiontheorem}
Let $R$ be a PID of characteristic zero. Then for any $c\in R$, $H^\DP_{1,c}(R)$ is a free $R$-module.
\end{thm}
\begin{proof}
For now, let $R$ be an arbitrary domain of characteristic zero. For any $k\geq 0$, consider the polynomials $D_k^+(t) = \prod^{k-1}_{i=0}(2t-i-2cp_i)$ and $D_k^-(t)=\prod^{k-1}_{i=0}(2t+1-i-2cp_{i+1})$ where $p_i=0$ if $i$ is even and $1$ otherwise. Note that $D^k\e_+x^{2n}=D^+_k(n)x^{2n-k}$ and $D^k\e_-x^{2n+1}=D^-_k(n)x^{2n+1-k}$. Now consider the $R$-modules
\[H^+[n] = \begin{cases}R[2t]& n\geq 0 \\ D^+_{-n}(t)R[2t] & n < 0\end{cases}\quad\rm{ and }\quad H^-[n] = \begin{cases}R[2t+1]& n\geq 0 \\ D^-_{-n}(t)\Z[2t+1,\ell] & n < 0\end{cases}.\]
\textit{Note:} We are aware that $R[2t+1] = R[2t]$; this distinction is purely to motivate the connection between these sets and $H_{1,c}(R)$.

We have a $\Z$-grading on $H_{1,c}(R)$, given by $D\mapsto -1$, $x\mapsto 1$ and $s\mapsto 0$. By the PBW theorem, there is an isomorphism $H_{1,c}(R) \to \bigoplus_{n\in\Z} P[n]$ where $P[n]$ is the module of homogeneous elements of $H_{1,c}(R)$ of degree $n$. For all $Q\in H_{1,c}(R)$, we have $Q=Q\e_++Q\e_-$ and $H_{1,c}(R)\e_+\cap H_{1,c}(R)\e_- = \{0\}$, so it follows that $P[n]=P[n]\e_+\oplus P[n]\e_-$. We claim that $\psi_n^+ : H^+[n] \to P[n]\e_+$ and $\psi_n^- : H^-[n] \to P[n]\e_-$ are isomorphisms, where $\psi_n^\pm$ sends $f(t)$ to the operator which maps $x^t$ to $\e_\pm f(t) x^{n+t}$. Note that this operator acts by zero on odd powers of $x$ in the $\e_+$ case, and by zero on even powers of $x$ in the $\e_-$ case. $\Ima(\psi_n^\pm)\subset P[n]\e_\pm$ and the surjectivity of these maps follows by a similar argument to the proof of Proposition~\ref{isomgrad}, and from the fact that $Q\in P[n]\e_\pm$ for $n<0$ implies that $Q=LD^{-n}\e_\pm$ for some $L\in P[0]$. Combining these maps gives an isomorphism of $R$-modules:
\[\psi : \bigoplus_{n\in\Z}\left(H^+[n]\oplus H^-[n]\right) \isom H_{1,c}(R).\]
We can consider this direct sum as a subring of $\End_{R}(R[x])$ given by the action of $H^+[n]\oplus H^-[n]$ on $x^n$, defined by $(f^+,f^-)x^{2t} = f^+(t)x^{2t+n}$ and $(f^+,f^-)x^{2t+1}=f^-(t)x^{2t+1+n}$. Note that by Proposition~\ref{directsumdividedpowers} and the argument used in Proposition~\ref{D_Zspan}, there is an induced isomorphism:
\[\psi^\DP : \bigoplus_{n\in\Z}\left(\Int_{R}(H^+[n])\oplus \Int_{R}(H^-[n])\right) \isom H_{1,c}^\DP(R) \ . \]
So to understand $H_{1,c}^\DP(R)$, it suffices to understand $\Int_R(H^\pm[n])$. By assumption, $R$ is a PID. Since $\Int_R(H^\pm[n])$ is a submodule of a free module, $R[t]$, it is free. By the isomorphism $\psi^\DP$, it follows that $H_{1,c}^\DP(R)$ is free as well.
\end{proof}

\begin{prop}\label{basistheoremforpid}
Let $R$ be a PID, and fix some $c\in R$. Then there exist coefficients $\alpha^\pm_{i,j,k}\in R$ and integers $d^\pm_{i,j} \in \Z_{\geq 1}$ yielding operators
\[\Delta^\pm_{k_1,k_2} = \begin{cases}\dfrac{D^{k_1}\sum^{k_2-1}_{i=0}\alpha_{i,k_1,k_2}^\pm(L^\pm)^i}{d^\pm_{k_1,k_2}}\e_\pm &\text{ if }k_1>0,\\
\dfrac{\prod_{i=0}^{k_2-1}(L^\pm-2i)}{2^{k_2}k_2!}\e_\pm &\text{ if }k_1=0,\end{cases}\]
where $L^+=xD$ and $L^-=xD+2c-1$. Then, the set $\{\Delta^\pm_{n,k}, x^{n+1}\Delta_{0,k}^\pm\}_{n,k\geq 0}$ is a basis for $H_{1,c}^\DP(R)$.
\end{prop}

\begin{proof}
To obtain this basis for $H_{1,c}^\DP(R)$, we first construct a basis for $\bigoplus_{n\in \Z}(\Int_R(H^+[n])\oplus \Int_R(H^-[n]))$. First suppose $n\geq 0$. In this case, $H^\pm[n]=R[2t]$, and so the binomial coefficients $\{{t\choose k}\}_{k\geq 0}$ form a basis for $\Int_R(H^\pm[n])$. Since for every $k\geq 0$, $\restr{\psi^\DP}{\Int_R(H^\pm[n])}\left({t\choose k}\right) = x^n\Delta^\pm_{0,k}$, it follows that $\{x^{n}\Delta_{0,k}^\pm\}_{n,k\geq 0}$ spans the set of non-negatively graded operators in $H_{1,c}^\DP(R)$. 

Now suppose $n<0$. It follows that $\Int_R(H^\pm[n])$ has a basis of the form $\left\{\frac{1}{d^\pm_{-n, k}}D^\pm_{-n}(t)\sum^{k-1}_{i=0}\alpha_{i,-n,k}^\pm (2t)^i\right\}_{k\geq 0}$. Since
\[\restr{\psi^\DP}{\Int_R(H^\pm[n])}\left(\frac{1}{d^\pm_{-n,k}}D^\pm_{-n}(t)\sum^{k-1}_{i=0}\alpha^\pm_{i,-n,k}(2t)^i\right) = \Delta^\pm_{-n,k}.\]
It follows that $\{\Delta_{n,k}^\pm\}_{n,k\geq 0}$ spans the set of negatively graded operators in $H_{1,c}^\DP(R)$, completing the proof. 
\end{proof} 

\subsection{Basis for $H_{1,c}^\DP(R[c])$}
In this section, we prove a similar result for $H_{1,c}^\DP(R[c])$. In this case, we can even construct an explicit basis for $H_{1,c}^\DP(R[c])$ as an $R[c]$-module.

\begin{thm}\label{basistheorem}
For any integers $k_1,k_2\geq 0$, consider the operators
\begin{itemize}
    \item $\displaystyle\Delta^+_{k_1,k_2} = \frac{D^{k_1}\prod^{k_2-1}_{i=0}(xD-2(i+m_1(k_1)))}{2^{m_1(k_1)+k_2} (m_1(k_1)+k_2)!}\e_+,$
    \item $\displaystyle \Delta^-_{k_1,k_2} = \frac{D^{k_1}\prod^{k_2-1}_{i=0}(xD+2c-1-2(i+m_0(k_1)))}{2^{m_0(k_1)+k_2} (m_0(k_1)+k_2)!}\e_-,$
\end{itemize}
where $m_\delta(k_1) = \floor*{\frac{k_1+\delta}{2}}$ for $\delta=0,1$. Then the set $\{\Delta_{n,k}^\pm, x^{n+1}\Delta_{0,k}^\pm\}_{n,k\geq 0}$ is an $R[c]$-basis for $H_{1,c}^\DP(R[c])$.
\end{thm}

Recall in the proof of Theorem~\ref{freenessofdividedpowerextensiontheorem}, we proved that $H_{1,c}^\DP(R)$ is isomorphic to the direct sum $\bigoplus_{n\in\Z}\left(\Int_{R}(H^+[n])\oplus \Int_{R}(H^-[n])\right)$ for any domain $R$. To prove Theorem~\ref{basistheorem}, we make use of this fact by constructing a basis for $\Int_R(H^\pm[n])$.

\begin{prop}
$(\psi^\DP)^{-1}\left(\{\Delta_{n,k}^\pm, x^{n+1}\Delta_{0,k}^\pm\}_{n,k\geq 0}\right)$ is a basis for $\bigoplus_{n\in\Z}\left(\Int_{R[c]}(H^+[n])\oplus \Int_{R[c]}(H^-[n])\right)$ as an $R[c]$-module.
\end{prop}
\begin{proof}
For any $k\geq 0$, consider the polynomials
\[L_k^+(t) = \prod^{m_0(k)-1}_{i=0}(2t-2i-1-2c), \quad\text{ and }\quad L_k^-(t)=\prod^{m_1(k)-1}_{i=0}(2t-2i+1-2c).\]
Borrowing notation from the proof of Theorem~\ref{freenessofdividedpowerextensiontheorem}, note that $D_k^+(t)=2^{m_1(k)}m_1(k)!L_k^+(t){{t}\choose {m_1(k)}}$ and $D_k^-(t)=2^{m_0(k)}m_0(k)!L_k^-(t){{t}\choose {m_0(k)}}$. Note that the statement of the Proposition is equivalent to the following four statements:
\begin{enumerate}
    \item The set $\left\{t\choose k\right\}_{k\geq 0}$ is an $R[c]$-basis for $\Int_{R[c]}(H^+[n])$ for $n\geq 0$.
    \item The set $\left\{L^+_{-n}{t\choose{k+m_1(-n)}}\right\}_{k\geq 0}$ is an $R[c]$-basis for $\Int_{R[c]}(H^+[n])$ for $n<0$.
    \item The set $\left\{t\choose k\right\}_{k\geq 0}$ is an $R[c]$-basis for $\Int_{R[c]}(H^-[n])$ for $n\geq 0$.
    \item The set $\left\{L^-_{-n}{t\choose{k+m_0(-n)}}\right\}_{k\geq 0}$ is an $R[c]$-basis for $\Int_{R[c]}(H^-[n])$ for $n<0$.
\end{enumerate}
We will only prove (1) and (2), since (3) and (4) are proved similarly. For (1), assume $n\geq 0$ and let $f(2t)\in R[2t]$. By induction, we can find coefficients $\alpha_k\in R[c]$ such that $f(t)=\sum_{k\geq 0}\alpha_k \prod^{k-1}_{i=0}(t-2i)$. Then $f(2t) = \sum_{k\geq 0}\alpha_k 2^kk!{t\choose k}$. By Lemma~\ref{intvaldivlemma}, if $d|f(2n)$ for all $n$ then $d|\alpha_k2^kk!$. This means that $\frac{f(t)}{d} = \sum_{k\geq 0}\frac{\alpha_k 2^k k!}{d}{t\choose k}$. Since $H^+[n]=R[2t]$ when $n\geq 0$, it follows that $\left\{t\choose k\right\}_{k\geq 0}$ is a basis for $\Int_{R[c]}(H^+[n])$.

For (2), assume $n<0$ and let $m = m_1(-n)$. Note that $D^+_{-n}(t) = L_{-n}^+(t)\prod^{m-1}_{i=0}(2t-2i)$. Let $f(2t)\in R[2t]$ be arbitrary and suppose $d|L_{-n}^+(t)\prod^{m-1}_{i=0}(2t-2i)f(t)$ since $L_{-n}^+(t)$ is a primitive polynomial, it follows that $d|\prod^{m-1}_{i=0}(2t-2i)f(t)$. Writing $f(t)=\sum_{j\geq 0}\alpha_jj!{t-m\choose j}$ we have $d|\prod^{m-1}_{i=0}(2t-2i)f(t)=\sum_{j\geq 0}2^{m+j}(m+j)!\alpha_j {t\choose {m+j}}$. By Lemma~\ref{intvaldivlemma},
\[\frac{D_{-n}^+(t)f(t)}{d} = \frac{L_{-n}^+(t)\prod^{m-1}_{i=0}(2t-2i)f(t)}{d}=\sum_{j\geq 0}\frac{2^{m+j}(m+j)!\alpha_j}{d}L_{-n}^+(t){t\choose{m+j}}.\]
and the claim follows, since $H^+[n]=D_{-n}^+(t)R[2t]$ when $n<0$.
\end{proof}

The above proposition immediately implies Theorem~\ref{basistheorem}.

\subsection{Hilbert Series for $H_{1,c}^\DP(R)$}
\begin{defn}
Let $M$ be a module over a domain $R$ and suppose we have a filtration $M=\bigcup_{i\geq 0}M_i$. Let $\rm{gr}(M)$ be the associated graded module of $M$ with respect to the filtration, i.e. $\rm{gr}(M)=M_0\oplus \bigoplus_{i\geq 1}(M_{i}/M_{i-1})$. Let $\rm{gr}_n(M)$ be the $n$-th graded component of $\rm{gr}(M)$. The Hilbert series of $M$ is defined as 
\[\mathbf{HS}_M(z) = \sum_{n\geq 0}\dim_R(\rm{gr}_n(M))z^n.\]
\end{defn}

In the following proposition, we show that the Hilbert series of the rational Cherednik algebra of type $A_1$ remains unchanged after the divided power extension construction. 

\begin{prop}
Let $R$ be a principal ideal domain. Then:
\begin{enumerate}
    \item $\mathbf{HS}_{H_{1,c}(R)}(z) = \frac{2}{(1-z)^2}$.
    \item $\mathbf{HS}_{H_{1,c}^\DP(R[c])}(z) = \frac{2}{(1-z)^2}$.
    \item For any $c\in R$, $\mathbf{HS}_{H_{1,c}^\DP(R)}(z) = \frac{2}{(1-z)^2}$.
\end{enumerate}
\end{prop}

\begin{proof}
(1) immediately follows from the PBW Theorem, since $H_{1,c}(R)$ is generated by elements of the form $x^{\ell}D^k\e_\pm$. This implies that $\dim_R(\rm{gr}_n(H_{1,c}(R)))=2(n+1)$. (2) follows from a similar argument, since by Theorem~\ref{basistheorem}, $\dim_{R[c]}(\rm{gr}_n(H_{1,c}(R[c])))=2(n+1)$. (3) is the same as (2), since Proposition~\ref{basistheoremforpid} shows that the basis for $H_{1,c}^\DP(R)$ has the same degree as the basis for $H_{1,c}^\DP(R[c])$.
\end{proof}

\subsection{The Lie Algebra $\frak{sl}_2$}
\begin{defn}
A triple of operators $E,H,F$ is said to be an $\fr{sl}_2$-triple if:
\begin{itemize}
    \item $[H,E]=2E$
    \item $[H,F]=-2F$
    \item $[E,F]=H$
\end{itemize}
\end{defn}

\begin{prop}
In $H_{1,c}(R[c])$ let $H=(xD+\frac{1-2c}{2})\e_+$, $E=-\frac{1}{2}x^2\e_+$, and $F=\frac{1}{2}D^2\e_+$. Then $E,H,F$ form an $\fr{sl}_2$-triple. It follows that $\e_+H_{1,c}(R[c])\e_+$ is isomorphic to a quotient of $\mathcal{U}(\fr{sl}_2)$ by the central character $\langle C+\frac{(1-2c)(3+2c)}{8}\rangle$, where $C$ is the Casimir operator $C=EF+FE+\frac{H^2}{2}$.
\end{prop}

This map suggests a divided power structure on this quotient of $\mathcal U(\fr{sl}_2)$. An immediate corollary to Theorem~\ref{freenessofdividedpowerextensiontheorem} states:
\begin{cor}\label{sphericalcorollary}
The set $\{\Delta^+_{2n,k}, x^{2n+2}\Delta^+_{0,k}\}_{n,k\geq 0}$ is an $R[c]$-basis for $\e_+ H_{1,c}^\DP(R[c])\e_+$.
\end{cor}
Writing this basis in terms of the $\fr{sl}_2$-triple gives us a basis for a divided power structure on $\mathcal U(\fr{sl}_2)$. Let 
\[\Sigma_{a,b,c} = \frac{(-2E)^a (2F)^b \prod^{c-1}_{i=0}\left( H- \frac{1-2c}{2}-2(i+m_1(2b)\right)}{2^{m_1(2b)+c}(m_1(2b)+c)!}\in \mathcal U(\fr{sl}_2(\Q)).\]
Then the set $\{\ \Sigma_{0,n,k}, \Sigma_{n+1,0,k}\}_{n,k\geq 0}$ is a basis for a divided power structure on a quotient of $\mathcal{U}(\fr{sl}_2)$.

\textit{Note:} This basis of divided powers is different from the basis given in \cite{jantzen2007representations}. Indeed the basis given there is symmetric, containing both divided powers of $E$ and $F$. Our divided power extension contains no divided powers of $E$ (indeed the denominator above does not depend on $a$ at all), but it has more divided powers of $F$.

\section{Abstract Construction of $H_{1,c}^\DP(R)$}\label{abstractconstructionsection}
In this section, we prove Theorem~\ref{abstractconstructiontheorem} which takes some setup to properly state.
\subsection{Grothendieck Differential Operators}
Before stating the main theorem, we recall a purely algebraic notion of differential operators due to Grothendieck. The results from this section can be found in \cite{diffoperatorsutah}.

\begin{defn}[Grothendieck Differential Operators] Let $R\subset A$ be a pair of commutative rings. For any $a\in A$, let $\overline{a}$ be the multiplication by $a$ operator on $A$. We define the $R$-linear differential operators on $A$ of order at most $i$, denoted $\Diff_R(A)^i$ inductively in $i$.
\begin{itemize}
    \item $\Diff_R(A)^0 = \Hom_A(A,A) = \{ \overline{a} : a\in A\}$
    \item $\Diff_R(A)^i = \{f\in \Hom_R(A,A) : [f,\overline{a}]\in \Diff_R(A)^{i-1}, \forall a\in A\}$
\end{itemize}
Let $\Diff_R(A) = \bigcup^\infty_{i=0}\Diff^i_R(A)\subset \Hom_R(A,A)$ be the algebra of differential operators of $A$ over $R$. When $R$ is clear, we simply write $\Diff(A)$ to denote $\Diff_R(A)$.
\end{defn}

For the results in Section~\ref{abstractconstruction}, it suffices to consider differential operators of polynomial algebra. The following results describe the structure of the ring of differential operators completely.

\begin{defn}
For any $\lambda \in \N^n$ let $\partial^\lambda$ be the Hasse derivative, i.e. the $R$-linear operator on $R[x_1,\ldots,x_n]$ given on the basis by
\[\partial^\lambda(x_1^{\beta_1}\cdots x_n^{\beta_n}) = {\beta_1\choose \lambda_1}\cdots {\beta_n\choose \lambda_n}x_1^{\beta_1-\lambda_1}\cdots x_n^{\beta_n-\lambda_n}.\]
\end{defn}

In rings where $\lambda_1!\cdots\lambda_n!\in R^\times$, the operator $\partial^\lambda$ is simply $\partial^\lambda = \frac{1}{\lambda_1!\cdots\lambda_n!}\frac{\partial^{\lambda_1}}{\partial x_1^{\lambda_1}}\cdots \frac{\partial^{\lambda_n}}{\partial x_n^{\lambda_n}}$.

\begin{prop}
Let $A=R[x_1,\ldots,x_n]$. Then $\Diff_R(A) = \bigoplus_{\lambda\in\N^n}A\partial^\lambda$, where multiplication is given by composition of operators.
\end{prop}

Since we are dealing with differential operators defined on a punctured line, we need to consider rings of differential operators over localized polynomial rings as well.

\begin{prop}
Let $R\subset A$ be rings where $A$ is of finitely generated over $R$. Let $W\subset A$ be a multiplicative subset. Then $W^{-1}\Diff^i_R(A) \cong \Diff^i_R(W^{-1}A)$.
\end{prop}

\begin{cor} $\Diff_R(R[x^{\pm 1}_1,\ldots,x_n^{\pm 1}]) =  \bigoplus_{\lambda\in\N^n}R[{x^{\pm 1}}_1,\ldots, {x^{\pm 1}}_n]\partial^\lambda$, where multiplication is given by composition of operators.
\end{cor}
\subsection{Abstract Construction}\label{abstractconstruction}
In this section, we would like to naturally define the ring $H_{1,c}^\DP (R[c])$ as a space of differential operators preserving some sets of the form $x^k|x|^rR[x]$, for some $k\in \Z$ and $r\in R$. Here $|x|^r$ is fixed by the action of $\fr S_2$, and $\frac{\partial}{\partial x}|x|^r =rx|x|^{r-2}$. We will denote this ring as $\mathcal H_{c}(R)$, and its definition should be purely algebraic, similar to the definition of $\Diff_R(A)$. First, we need a nice space of differential operators to work in.

\begin{defn}
For any domain of characteristic zero $R$, let $\fr{D}(R)$ be the ring
\[\fr D(R) = \Diff_{R}(R[{x^{\pm 1}}]\otimes_R \fr{S}(R))\]
where $\fr S (R) = R\e_+\oplus R\e_-$ is the ring acting on $R[{x^{\pm 1}}]$ the canonical way. Note that $\fr D(R) = \left(\Diff_R(R[{x^{\pm 1}}])\rtimes R[\fr S_2]\right)^\DP$.
\end{defn}  

Our main theorem of the section can then be stated:

\begin{thm}\label{abstractconstructiontheorem}
For a domain of characteristic zero $R$ and $c\in R$, consider
\[\mathcal H_c(R) = \{Q\in \fr D(R) : Q \rm{ fixes } R[x]\rm{ and }x^{-1}|x|^{1+2c}R[x]\}\]
Then, $\mathcal H_c(R)\cong H_{1,c}^\DP(R)$ if $c\not\in \frac{1}{2}+\Z$.
\end{thm}

To prove this theorem, it is useful to decompose $H_{1,c}(R[c])$ in the following way:
\[H_{1,c}(R[c]) \cong \e_+ H_{1,c}(R[c])\e_+ \oplus \e_+ H_{1,c}(R[c])\e_- \oplus \e_- H_{1,c}(R[c])\e_+ \oplus \e_- H_{1,c}(R[c])\e_- \]
Expressing each of these summands in a similar way to $\mathcal H_c(R)$ helps with the proof. Note that $\e_\pm H_{1,c}^\DP(R) \e_\pm = (\e_\pm H_{1,c}(R) \e_\pm)^\DP$, where $\e_\pm$ can be either $\e_+$ or $\e_-$.
\begin{defn}
For a domain of characteristic zero $R$ and $c\in R$, consider the following sets:
\begin{itemize}
    \item $\mathcal B_c(R) = \{ Q\in \e_+\fr D(R)\e_+ : Q \text{ fixes } R[x] \text{ and } Q \text{ fixes }|x|^{1+2c}R[x]\}$.
    \item $\overline{\mathcal B_c(R)} = \{ Q\in \e_-\fr D(R)\e_- : x^{-1}Qx \text{ fixes } R[x] \text{ and } xQx^{-1} \text{ fixes }|x|^{1+2c}R[x]\}$.
    \item  $\mathcal A_c(R) = \{ Q\in \e_-\fr D(R)\e_+ : Q \text{ fixes } R[x] \text{ and } xQ \text{ fixes }|x|^{1+2c}R[x]\}$.
    \item  $\overline{\mathcal A_c(R)} = \{ Q\in \e_+\fr D(R)\e_- : Qx \text{ fixes } R[x] \text{ and } xQx^{-1} \text{ fixes }|x|^{1+2c}R[x]\}$.
\end{itemize}
\end{defn}

\begin{prop}\label{fourcomonentsprop} If $c\not\in\frac{1}{2}+\Z$ then $\mathcal B_c(R) = \e_+ H_{1,c}^\DP(R)\e_+$, $\overline{\mathcal B_c(R)} = \e_- H_{1,c}^\DP(R)\e_-$, $\mathcal A_c(R) = \e_- H_{1,c}^\DP(R) \e_+$, and  $\overline{\mathcal A_c(R)} = \e_+H_{1,c}^\DP(R)\e_-$.
\end{prop}

\begin{proof}
We will only prove the first equality, $\mathcal B_c(R) = \e_+H_{1,c}^\DP(R)\e_+$, the rest follow similarly. First, we show that $\e_+ H_{1,c}^\DP(R)\e_+\subset\mathcal B_c(R)$. Let $Q\in \e_+H_{1,c}^\DP(R)\e_+$ be some operator. If we write $Q=\sum_{n\in \Z}Q_n$, where $\deg Q_n = n$, it suffices to check that $Q_n\in \mathcal B_c(R)$. So without loss of generality, assume $Q$ is graded of degree $n$. If $n\geq 0$, clearly $Q\in \mathcal B_c(R)$. If $n<0$, then $Q$ can be expressed as $Q=\e_+LD^{-n}\e_+/d$ for some $L$ of degree 0 and $d\in \Z$. 

To check that $Q$ fixes $R[x]$ and $|x|^{1+2c}R[x]$, it suffices to check the action of $Q$ on monomials. To start, let's consider the action of $Q$ on $x^k$ for some $k\geq 0$. If $k$ is odd, $Qx^k=0\in R[x]$. If $k$ is even, there are two cases. If $k\geq -n$, then $Qx^k = \lambda D^+_{-n}(k) x^{k+n}/d\in R[x]$ since $k+n\geq 0$ (Recall notation from the proof of Theorem~\ref{freenessofdividedpowerextensiontheorem}). If $k<-n$, note that $D_{-n}^+(k)=0$, so $Qx^k=0\in R[x]$. A similar thing happens for $|x|^{1+2c}R[x]$, since $D_{-n}^+(k+1+2c)=0$ for even $k<-n$. This shows that $\e_+ H_{1,c}^\DP(R)\e_+\subset\mathcal B_c(R)$. 

Next, we show that $\mathcal B_c(R)\subset \e_+H_{1,c}^\DP(R)\e_+$. As before, we can assume that $Q$ is graded of degree $n$. Let $f(t)$ be the polynomial representing the action of $Q$, i.e. $Qx^k = f(k)x^{k+n}$. If $n\geq 0$, write $f(t) = \sum_{j\geq 0}\alpha_j t^j$ for some $\alpha_j\in R\otimes \Q$. This tensor product with $\Q$ arises from the fact that $\e_+\fr D(R)\e_+=\left(\e_+\Diff_R(R[{x^{\pm 1}}]\e_+\right)^\DP$, hence operators might have coefficients in $R\otimes \Q$. Then
\[Q =\e_+ x^n\sum_{j\geq 0}\alpha_j (xD)^j\e_+ \in \e_+ H_{1,c}^\DP(R)\e_+.\]
Now suppose $n<0$. Notice that $f(k)=f(k+1+2c)=0$ for all even $k$ satisfying $0\leq k<-n$, so \linebreak $\prod^{-n/2-1}_{j=0}(t-2j)(t-2j-1-2c)$ divides $f(t)$. This is exactly the action of the Dunkl operator $D^{-n}\e_+$. Also note that this depends on the fact that $c\not\in \frac{1}{2}+\Z$, otherwise, the linear factors could overlap. Let $L(t)=\sum_{j\geq 0}\beta_j t^j$ be the quotient of this division for some $\beta_j\in R\otimes \Q$. Then
\[Q=\e_+D^{-n}\sum_{j\geq 0}\beta_j (xD)^j\e_+,\] completing the proof.
\end{proof}
\begin{prop}\label{fourcomonentsprop2}
$\mathcal H_c(R) \cong \mathcal B_c(R) \oplus \overline{\mathcal B_c(R)} \oplus \mathcal A_c(R) \oplus \overline{\mathcal A_c(R)}.$
\end{prop}
\begin{proof}
Let $H=\mathcal B_c(R) \oplus\overline{\mathcal B_c(R)} \oplus \mathcal A_c(R)\oplus \overline{\mathcal A_c(R)}$. Consider both $\mathcal H_c(R)$ and $H$ as subrings of $\End_R(R[x])$. First we show that $H\subset \mathcal H_c(R)$. Let $Q\in H$ be a graded operator, say $Q= \e_+ Q\e_+ + \e_-Q\e_-+ \e_+Q\e_- + \e_-Q\e_+$. First we show that $Q$ fixes $R[x]$. By Proposition~\ref{fourcomonentsprop},
\begin{align*}
    Q(R[x]) &= \e_+ Q\e_+(R[x]) + \e_- Q\e_-(R[x]) + \e_+Q\e_-(R[x]) + \e_-Q\e_+(R[x])\\
    &=\e_+ Q\e_+(R[x]) + \e_- Q\e_-(xR[x]) + \e_+ Q\e_-(xR[x]) + \e_- Q\e_+(R[x])\\
    &\subset R[x]+R[x]+R[x]+R[x] \subset R[x]
\end{align*}
because $x^{-1}\e_- Q\e_-(xR[x])\subset R[x]$ implies that $\e_-Q\e_-(xR[x])\subset R[x]$. Let $y=x^{-1}|x|^{1+2c}$. By Proposition~\ref{fourcomonentsprop}, we have
\begin{align*}
    Q(yR[x]) &= \e_+ Q\e_+(yR[x]) + \e_- Q\e_-(yR[x]) + \e_+Q\e_-(yR[x]) + \e_-Q\e_+(yR[x])\\
    &=\e_+ Q\e_+(xyR[x]) + \e_- Q\e_-(yR[x]) + \e_+ Q\e_-(yR[x]) + \e_-Q\e_+(xyR[x])\\
    &\subset yR[x] + yR[x]+yR[x]+yR[x] \subset yR[x].
\end{align*}
So $H\subset \mathcal H_c(R)$. To show that $\mathcal H_c(R)\subset H$, suppose $Q\in \mathcal H_c(R)$ is some graded operator. If $\deg Q$ is even, then $Q=\e_+Q\e_+ + \e_-Q\e_-$. Since $Q(R[x])\subset R[x]$, $Q(R[x]) = \e_+Q\e_+(R[x])+\e_-Q\e_-(R[x]) \subset R[x]$, and $\e_+ Q\e_+$, $\e_-Q\e_-$ act non-trivially on only even and odd degrees of $x$ respectively, it follows that $\e_+ Q\e_+(R[x])\subset R[x]$ and $x^{-1}\e_-Q\e_-x(R[x])\subset R[x]$. Similarly, we can deduce that $\e_+ Q\e_+(|x|^{1+2c}R[x])\subset |x|^{1+2c}R[x]$ and $\e_-Q\e_-(x^{-1}|x|^{1+2c}R[x])\subset x^{-1}|x|^{1+2c}R[x]$. So $\e_+Q\e_+\in \mathcal B_c(R)$ and $\e_-Q\e_-\in \overline{\mathcal B_c(R)}$. Similarly, in the case when $\deg Q$ is odd we can show that $\e_+Q\e_-\in \overline{\mathcal A_c(R)}$ and $\e_-Q\e_+\in \mathcal A_c(R)$. This shows that $\mathcal H_c(R)\subset H$, completing the proof.
\end{proof}

To prove Theorem~\ref{abstractconstructiontheorem}, note that by Proposition~\ref{directsumdividedpowers}, $H_{1,c}^\DP(R)\cong \e_+H_{1,c}^\DP(R)\e_+\oplus \e_-H_{1,c}^\DP(R)\e_+\oplus \e_+H_{1,c}^\DP(R)\e_-\oplus \e_-H_{1,c}^\DP(R)\e_+$. If $c\not\in\frac{1}{2}+\Z$, by Proposition~\ref{fourcomonentsprop} and Proposition~\ref{fourcomonentsprop2}, we have
\begin{align*}\mathcal H_c(R) &\cong \mathcal B_c(R) \oplus \overline{\mathcal B_c(R)} \oplus \mathcal A_c(R) \oplus \overline{\mathcal A_c(R)} \\&\cong \e_+H_{1,c}^\DP(R)\e_+\oplus \e_-H_{1,c}^\DP(R)\e_+\oplus \e_+H_{1,c}^\DP(R)\e_-\oplus \e_-H_{1,c}^\DP(R)\e_+\\
&\cong H_{1,c}^\DP(R).
\end{align*}
This concludes the proof.

\subsection{The case $c\in \frac{1}{2}+\Z$}
Interestingly, the case $c\in \frac{1}{2}+\Z$ appears throughout the theory of Cherednik algebras. In the case of our construction, this exception appears because the polynomial representing the action of the Dunkl operator has multiplicity two zeroes, when our construction can only encode multiplicity one zeroes. A future direction would be to extend our construction of $\mathcal H_c(R)$ so that it works even when $c\in \frac{1}{2}+\Z$. Pavel Etingof suggested that the construction should preserve an infinite family of subsets of functions in $x$ involving $|x|$ which converge to some set of functions involving $|x|$ and $\log(x)$ as $c$ approaches a half-integer. This is useful by the following proposition:

\begin{prop} For $f(t)\in \Z[t]$ and $F\in \Z[x, \partial_x]$ the operator mapping $x^n$ to $f(n)x^{n+d}$ for some $d\in \Z$,
\[F(x^n\log(x)) = \frac{df}{dt}(n)x^{n+d}+f(n)x^{n+d}\log(x).\]
Here we let $\partial_x(\log x) = \frac{1}{x}$.
\end{prop}

So using $\log(x)$, we can encode information about the multiplicity-two roots about the polynomial which represents the action of the operator. Since the Dunkl operator has roots of at most multiplicity two, there is a construction which should work in all cases.

\subsection*{Acknowledgements}
This project was done under the MIT PRIMES-USA program, which we would like to thank for this opportunity. We would also like to thank Professor Pavel Etingof for suggesting this project, and for useful discussions on the research.

\bibliography{main}
\end{document}